\documentclass[12pt,english]{amsart}
\usepackage[T1]{fontenc}
\usepackage[latin9]{inputenc}
\usepackage{geometry}
\usepackage{color}
\usepackage{textcomp}
\usepackage{amstext}
\usepackage{amsthm}
\usepackage{amssymb}
\usepackage{graphicx}

\makeatletter

\newcommand*\LyXZeroWidthSpace{\hspace{0pt}}
\providecommand{\tabularnewline}{\\}

\numberwithin{equation}{section}
\numberwithin{figure}{section}
\theoremstyle{plain}
\newtheorem{thm}{\protect\theoremname}
  \theoremstyle{remark}
  \newtheorem{rem}[thm]{\protect\remarkname}
  \theoremstyle{definition}
  \newtheorem{defn}[thm]{\protect\definitionname}
  \theoremstyle{plain}
  \newtheorem{lem}[thm]{\protect\lemmaname}

\makeatother

\usepackage{babel}
  \providecommand{\definitionname}{Definition}
  \providecommand{\lemmaname}{Lemma}
  \providecommand{\remarkname}{Remark}
\providecommand{\theoremname}{Theorem}

\begin{document}

\title{Convergent approaches for the Dirichlet Monge-Ampère problem}

\maketitle
\begin{center}\author{Hajri  Imen \footnote{Higher Institute of Applied Studies in Humanities of Mahdia,5121 Mahdia, Tunisia, Email:hajri.imene2017@gmail.com.} Fethi Ben Belgacem \footnote{Laboratory of partial differential equations (LR03ES04), ISIMM, University of Monastir, El Manar, TUNISIA. Email: fethi.benbelgacem@isimm.rnu.tn }}\end{center}
\begin{abstract}
In this article, we introduce and study three numerical methods for
the Dirichlet Monge-Ampère equation in two dimensions. The approaches
consist in considering new equivalent problems. The latter are discretized
by a wide stencil finite difference discretization and monotone schemes
are obtained. Hence, we apply the Barles-Souganidis theory to prove
the convergence of the schemes and the Damped Newtons method is used
to compute the solutions of the schemes. Finally, some numerical results
are illustrated. 
\end{abstract}


\section{Introduction}

We are interested in the numerical solution of the Monge-Ampère equation
with Dirichlet boundary condition
\begin{equation}
\textrm{(MAD)}\left\{ \begin{aligned}\mbox{det}\left(D^{2}u\left(x\right)\right)=f\left(x\right), & \textrm{ for }x\textrm{ in }\varOmega,\\
u\left(x\right)=\varphi\left(x\right), & \textrm{ for }x\textrm{ on }\partial\varOmega,\\
u\textrm{ is convex.}
\end{aligned}
\right.\label{eq:MAD}
\end{equation}

Where $\Omega$ is a convex bounded domain in $\mathbb{R}^{2},$ with
boundary $\partial\varOmega,$ $\left(D^{2}u\right),$ is the Hessian
of the function $u,$ $f\textrm{ and }\varphi$ are given functions.

We take the simplest boundary conditions. For more general operator
of Monge-Ampère and other boundary conditions, we mention for instance
\cite{Figali}. The convexity constraint is crucial for the (MAD).
It is required for the Monge-Ampère equation to be degenerate elliptic
and for (MAD) to have a unique solution. It is also needed for numerical
stability. The Monge-Ampère equation, has extensive applications,
it is strictly related to the \textquotedblleft prescribed Gauss curvature\textquotedblright{}
problem, see for instance \cite{Figali}. It appears also in affine
geometry, precisely, in the affine sphere problem and the affine maximal
surfaces problem, this was discussed in \cite{calabi,cheng yau,pogorelov,Trudingin1,trudinger2,trudinger3}.
Other applications appear in fluid mechanics, geometric optics, and
meteorology : for example, in semigeostrophic equations, the Monge-Ampère
equation is coupled with a transport equation, this is pointed out
in \cite{Figali}. The analysis of the regularity of the Monge-Ampere
equation is essential in the study of the regularity of the transp
ort problem. This, latter, has been employed in many areas. We only
briefly mention \cite{Delzano1,Delzano2,Budd} for mesh geneartion,\textcolor{black}{\cite{Rehman,Haker1,Haker2}}for
image registration, and \cite{Figali} for reflector design. Developing
an efficient numerical method has aroused a lot of interest, and large
standard techniques have been proposed. A first method to do so was
introduced in \cite{prusner} by using a discretization of the geometric
Alexandrov-Bakelman interpretation of solutions. Variational approaches
have been presented in \cite{Glowinski1,glowinski2}, more precisely,
the augmented Lagrangian approach and the least-squares approach.
But these methods needed more regularity than can be predicted for
solutions. A different approach was studied in \cite{Neilan1}, using
the vanishing moment method. The periodic case was treated in \cite{Gr=0000E9goire}. 

Although, the standard techniques, mentioned above, work well for
smooth solutions, and they fail for singular solutions, for more details,
see, for instance, the discussion in \cite{Ben Amou}. To overcome
these difficulties, we have to use the notion of viscosity solution
or Alexsandrov solution. In two dimension, a numerical method was
introduced in \cite{prusner}, which is geometric in nature, and converges
to the Alexsandrov solution. The method introduced in \cite{oberma08},,
in two dimension and improved in \cite{oberma11} for higher dimension,
uses the wide stencil scheme that converges to the viscosity solution,
which we briefly describe for this reason in the end of this section.

The following variant of the AM-GM inequalities, is the keystone of
our formulation introduced here.

For $A$ and $B$ two symmetric matrices, such that, $A,B\geq0.$
We have the following inequality 
\[
2\sqrt{\det\left(AB\right)}\leq Tr\left(AB\right).
\]

Where for symmetric matrices $M\geq0$ means $x^{T}Mx\geq0.$ 
\begin{rem}
We can deduce from the above inequality that for a smooth convex solution
$u$ of (\ref{eq:MAD}), one can deduce the following inequality 
\[
\Delta u-2\sqrt{f}\geq0.
\]
Let us define the function 
\[
\tilde{g}:=\Delta u-2\sqrt{f}.
\]
It is then straightforward to check that if $u$ is a smooth solution
of (\ref{eq:MAD}), then is indeed a solution of the linear Dirichlet
Poisson problem
\begin{equation}
\left(\mathcal{P}^{\tilde{g}}\right)\left\{ \begin{array}{c}
\Delta u=2\sqrt{f}+\tilde{g},\\
u_{|\Gamma}=\varphi,
\end{array}\right.\label{eq:Pgdilde}
\end{equation}
which can be easily descretized by any method of choice if the function
$\tilde{g}$ is known.

We finish this remark by mentioning that the convexity constraint
is essential to ensure uniqueness (for example, $u$ and $-u$ are
both solution of the Monge Ampère equation). For viscosity solution,
this constraint can be required by the equation 
\begin{equation}
\lambda_{1}\left(D^{2}u\right)\geq0,\label{eq:lamda1-1}
\end{equation}
 in the viscosity sense, see for instance \cite{ober08env,oberma08},
where $\lambda_{1}\left(D^{2}u\right)$ is the smallest eigenvalue
of the Hessian of $u$. However, for a twice continuously differentiable
function $u$, the convexity restriction is equivalent to requiring
that the eigenvalues \LyXZeroWidthSpace \LyXZeroWidthSpace of the
Hessian, $D^{2}u,$ are positives, which is approved by considering
the linear Poisson Dirichlet problem $\left(\mathcal{P}^{\tilde{g}}\right).$ 
\end{rem}
\textcolor{black}{The approaches that we follow, in the present paper,
are inspired by the idea developed in \cite{FBB} and} the wide stencil
finite difference discretization introduced in \cite{oberma08} and
\cite{oberma11} for viscosity solution of M-A equation in two and
higher dimensions that relies on a framework developed in {[}1{]}.
For clarity, we recall the full result in the next section.

\section{Viscosity solution and convergence theory of approximation schemes}

\subsection{Degenerate elliptic equations}

Let $F\left(x,r,p,X\right)$ be a continuous real valued function
defined on $\varOmega\times\mathbb{R}\times\mathbb{R}^{n}\times S^{n},$
with $S^{n}$ being the space of symmetric $n\times n$ matrices.
Consider the nonlinear, partial differential equation with Dirichlet
boundary conditions,
\[
\begin{cases}
F\left(x,u\left(x\right),Du\left(x\right),D^{2}u\left(x\right)\right)\left(x\right)=0 & \textrm{ for }x\textrm{ in }\varOmega\\
u\left(x\right)=g\left(x\right) & \textrm{ for }x\textrm{ in }\partial\varOmega.
\end{cases}
\]
Where $\varOmega$ is a domain in $\mathbb{R}^{n},$ $Du$ and $D^{2}u$
denote the gradient and Hessian of $u$, respectively.
\begin{defn}
\cite{oberma11}The equation $F$ is degenerate elliptic if
\[
F\left(x,r,p,X\right)\leq F\left(x,s,p,Y\right)\textrm{ whenever }r\leq s\textrm{ and }Y\leq X.
\]
Where $Y\leq X$ means that $Y-X$ is a nonnegative definite symmetric
matrix.
\end{defn}
The viscosity solution for the Monge-Ampère equation is defined in
\cite{oberma08}.
\begin{defn}
Let $u\in C\left(\varOmega\right)$ be convex and $f\geq0$ be continuous.
The function $u$ is a viscosity subsolution (supersolution) of the
Monge-Ampère equation in $\varOmega$ if whenever convex $\varphi\in C^{2}\left(\varOmega\right)$
and $x_{0}\in\varOmega$ are such that $\left(u-\varphi\right)\left(x\right)\leq\left(\geq\right)\left(u-\varphi\right)\left(x_{0}\right)$
for all $x$ in a neighborhood of $x_{0},$ then we must have 
\[
\det\left(D^{2}\phi\left(x_{0}\right)\right)\geq\left(\leq\right)f\left(x_{0}\right).
\]
 The function $u$ is a viscosity solution if it is both a viscosity
subsolution and supersolution.
\end{defn}
For the existence and uniqueness of viscosity solution for (\ref{eq:MAD}),
we mention the next result in \cite{Guti=0000E8rez},
\begin{thm}
Let $\varOmega\subseteq\mathbb{R}^{d}$ be abounded and strictly convex,
$g\in C\left(\partial\varOmega\right),$ $f\in C\left(\varOmega\right),$
with $f\geq0.$ Then there exists a unique convex viscosity solution
$u\in C\left(\bar{\varOmega}\right)$ of the problem (\ref{eq:MAD}).
\end{thm}
The advantage of considering viscosity solutions come from the following
fundamental theorem, obtained in \cite{Souganidis}, which gives conditions
for convergence of approximation schemes to viscosity solution.
\begin{thm}
(Convergence of Approximation Schemes). Consider a degenerate elliptic
equation, for which there exist unique viscosity solutions. A consistent,
stable approximation scheme converges uniformly on compact subsets
to the viscosity solution, provided it is monotone.
\end{thm}
By the previous theorem, we need just a way to build a monotone finite
difference schemes, which represents a new challenge. In the sequel,
we recall here the basic framework introduced in \cite{oberma06},
for building a monotone scheme. 

Firstly, a finite difference equation take the form
\[
F^{i}\left[u\right]=F^{i}\left(u_{i},u_{i}-u_{j}|_{i\neq j}\right).
\]
 We say that a scheme is degenerate elliptic if the following holds
\cite{oberma06}:
\begin{defn}
The scheme $F$ is degenerate elliptic if $F^{i}$ is non-decreasing
in each variable.
\end{defn}
We are now ready to present the following theorem in \cite{oberma06}:
\begin{thm}
\label{thm:fixpoint}Under mild analytic conditions, degenerate elliptic
schemes are monotone, and non-expansive in the uniform norm. The iteration
\begin{equation}
u^{m+1}=u^{m}+dtF\left(u^{m}\right),\label{eq:FixPoint}
\end{equation}
 is a contraction in $L^{\infty}$ provided $dt\leq K\left(F\right)^{-1},$
where $K\left(F\right)$ is the Lipschitz constant of the scheme,
regarded as a function from $\mathbb{R}^{N}\longrightarrow\mathbb{R}^{N}.$ 
\end{thm}
We end this paragraphre by the next result, proven in \cite{oberma06} 
\begin{thm}
A proper, locally Lipschitz continuous degenerate elliptic scheme
has a unique solution which is stable in the $l^{\infty}$ norm.
\end{thm}

\subsection{Wide stencil schemes}

\textcolor{black}{We finish this section by noting that wide stencil
schemes are required to build consistent, monotone schemes of degenerate
second order PDEs (see discussion in \cite{oberma08}). Wide stencil
schemes were built for the two-dimensional Monge-Ampère equation in
\cite{oberma08} and for the convex envelope in \cite{ober08env}.
Each approach considered here is a function of eigenvalues of the
Hessian. To fully discretize the equation (\ref{eq:Raylish-Labda})
for the eigenvalues of the Hessian on a finite difference grid, we
approximate the second derivatives by centered finite differences;
this is the spatial discretization, with parameter $h$. We consider
also a finite number of possible directions $\nu$ that lie on the
grid; this is the directional discretization, with parameter $d\theta.$
 The spatial resolution is improved by using more grid points, the
directional resolution is improved by increasing the size of the stencil.
So, a wide stencil is needed (see Fig \ref{Stencil 17})}

\begin{figure}
\includegraphics[width=7cm,height=5cm]{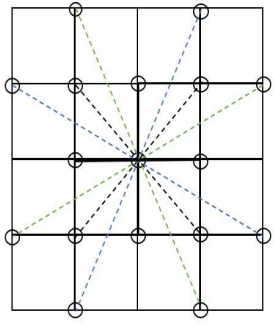}\label{Stencil 17}\caption{Grid for wide stencil 17 points, in two dimension}
\end{figure}

\section{First Formulation of the \textup{(MAD) in} two dimensions (method
A)}

\subsection{An equivalent problem}

Let us begin with a simple approach to illustrate the ideas. We can
rephrase, for instance, the (MAD) as the following:
\begin{equation}
\begin{cases}
\textrm{Find a positive function }g,\textrm{ such that}\\
\textrm{det}\left(D^{2}u^{g}\right)=\lambda_{1}\left[D^{2}u^{g}\right]\times\lambda_{2}\left[D^{2}u^{g}\right]=f,
\end{cases}\label{eq:P2}
\end{equation}
where: $u^{g}$ is the solution of 
\[
\begin{cases}
\Delta u^{g}=\lambda_{1}\left[D^{2}u^{g}\right]+\lambda_{2}\left[D^{2}u^{g}\right] & =2\sqrt{f}+g,\\
u_{|\Gamma}^{g} & =\varphi.
\end{cases}
\]

We are now ready to state a first example of our approches.
\begin{lem}
\label{lem:Lemma equiva}Provided the solution, u, of (\ref{eq:MAD})
is in $H^{2},$ there exists a unique positive function $\tilde{g}\in L^{2},$
such that $u=u^{\tilde{g}},$ where $u^{\tilde{g}}$ is the solution
of ($\mathcal{P}^{\tilde{g}}$). Conversely, if $u^{\bar{g}}$ is
solution of (\ref{eq:P2}) for some $\bar{g}>0,$ then $u^{\bar{g}}=u.$\end{lem}
\begin{proof}
Let $u$ be a solution of (\ref{eq:MAD}). From the above, one can
see easily, that $u=u^{\tilde{g}}.$ 

Conversely, if $u^{\bar{g}}$ is a solution of (\ref{eq:P2}), we
can clearly see that 
\[
\begin{cases}
\textrm{det}\left(D^{2}u^{\bar{g}}\right)=f>0\\
\Delta u^{\bar{g}}\geq0.
\end{cases}
\]
It follows that $u^{\bar{g}}$ is convex and satisfies (\ref{eq:MAD}).\end{proof}
\begin{rem}
We notice that according to the result in \cite{Siltakoski}, we have
equivalence of viscosity and weak solutions for the Poisson problem.
This motivates us to build a convergent scheme to the viscosity solution
of Poisson problem $\left(\mathcal{P}^{\tilde{g}}\right)$ through
the discretization of the (MAD) problem. The viscosity solution $u^{\tilde{g}}$
of $\left(\mathcal{P}^{\tilde{g}}\right)$ will be equivalent to the
weak solution of (MAD) problem in the distributional sense.
\end{rem}

\section{Discretization of the problem (\ref{eq:P2})}

Let us consider a regular and uniform cartesian grid, consider the
stencil at the reference point $x_{0}$ consist of the neighbors $x_{1},...,x_{N}$
(as in Figure 1). We can define $v_{i}$ in polar coordinates by 
\[
v_{i}=x_{i}-x_{0}=h_{i}v_{\theta_{i}}.
\]
 We assume that the stencil is symetric and we define the local spatial
resolution and the directional resolution respectively by 
\[
\bar{h}\left(x_{0}\right)={\displaystyle \max_{i}h_{i}}
\]
 and 
\[
d\theta={\displaystyle \max_{\theta\in\left[-\pi,\pi\right]}\min_{i}\left|\theta-\theta_{i}\right|}.
\]

\textcolor{black}{First, the problem (\ref{eq:P2}) can be written
as function of the eigenvalues of the Hessian. We will then start
by discretizing $\lambda_{1}$ and $\lambda_{2}.$ Hence by a simple
substitution we obtain the scheme for (\ref{eq:P2}). }

\textcolor{black}{We recall that the smallest and the largest eigenvalues
of a symmetric matrix can be represented respectively by the Rayleigh-Ritz
formula}

\begin{equation}
\lambda_{1}\left[D^{2}u\right]\left(x\right)=\min_{\theta}\frac{d^{2}u}{d\nu_{\theta}^{2}},\qquad\lambda_{2}\left[D^{2}u\right]\left(x\right)=\max_{\theta}\frac{d^{2}u}{d\nu_{\theta}^{2}},\label{eq:Raylish-Labda}
\end{equation}
where $\nu_{\theta}=\left(\cos\theta,\sin\theta\right)$is the unit
vector in the direction of the angle $\theta.$

This formula was used in \cite{oberma08} to build a monotone scheme
in two dimension for the (MAD).

We begin by building monotone schemes for $\lambda_{1}$ and $\lambda_{2}$
on a wide stencil uniform grid. These operators are used to give schemes
for all formulations in this paper. 

We discretize the eigenvalues of the Hessian by the following formula.
\begin{equation}
\lambda_{1}^{h,d\theta}\left[D^{2}u^{g}\right]\left(x\right)=\min_{i}\frac{u^{g}\left(x+v_{i}\right)-2u^{g}\left(x\right)+u^{g}\left(x-v_{i}\right)}{\left|v_{i}\right|^{2}}\label{eq:lambda1dis}
\end{equation}
and 
\begin{equation}
\lambda_{2}^{h,d\theta}\left[D^{2}u^{g}\right]\left(x\right)=\max_{i}\frac{u^{g}\left(x+v_{i}\right)-2u^{g}\left(x\right)+u^{g}\left(x-v_{i}\right)}{\left|v_{i}\right|^{2}}.\label{eq:lambda2dis}
\end{equation}

\begin{lem}
The schemes (\ref{eq:lambda1dis}) and (\ref{eq:lambda2dis}) are
degenerate elliptic.\end{lem}
\begin{proof}
We follow the same as in \cite{oberma08}.

Since each discrete second derivative in the direction $v_{i}$ is
the average of the terms which have the form $u_{j}^{g}-u_{i}^{g},$
they are non-decreasing in $u_{j}^{g}-u_{i}^{g}.$ Taking a minimum
(or maximum) of non-decreasing functions furnishes a non-decreasing
function.
\end{proof}
We finally substitute (\ref{eq:lambda1dis}) and (\ref{eq:lambda2dis})
in (\ref{eq:P2}) to obtain the wide stencil finite difference scheme
of (\ref{eq:P2})
\begin{equation}
\begin{cases}
\textrm{Find a positive function }g^{i},\textrm{ such that}\\
\lambda_{1}^{h,d\theta}\left[D^{2}u^{g^{i}}\right]\times\lambda_{2}^{h,d\theta}\left[D^{2}u^{g^{i}}\right]=f^{i},
\end{cases}\label{eq:schemeP2}
\end{equation}
with 
\[
\begin{cases}
\lambda_{1}^{h,d\theta}\left[D^{2}u^{g^{i}}\right]+\lambda_{2}^{h,d\theta}\left[D^{2}u^{g^{i}}\right] & =2\sqrt{f^{i}}+g^{i}.\\
u_{|\Gamma}^{g} & =\varphi.
\end{cases}
\]
Where $f^{i}=f\left(x_{i}\right)$ and $g^{i}=g\left(x_{i}\right).$ 
\begin{lem}
\label{lem:P2 degenerate}The scheme (\ref{eq:schemeP2}) is degenerate
elliptic.\end{lem}
\begin{proof}
From the properties of nondecreasing functions, obtained in \cite{oberma08},
\\ that if $G:\mathbb{R}^{2}\rightarrow\mathbb{R}$ is a nondecreasing
function, and if $F_{1}$ and $F_{2}$ are degenerate elliptic finite
difference schemes, then so is $F=G\left(F_{1},F_{2}\right).$ It
is also clear that the discretization $f^{i}=f\left(x_{i}\right)$
and $g^{i}=g\left(x_{i}\right)$ does affect the ordering properties.
We conclude that (\ref{eq:schemeP2}) is degenerate elliptic.
\end{proof}
\textcolor{black}{In the following, for simplicity, we omit the index
i when there is no ambiguity.}
\begin{defn}
We say the scheme $H^{h,d\theta}$ is consistent with the equation
(MAD) at $x_{0}$ if for every twice continuously differentiable function
$\varphi\left(x\right)$ defined in a neighborhood of $x_{0},$ $H^{h,d\theta}$$\left(\varphi\right)\left(x_{0}\right)\rightarrow H\left(\varphi\right)\left(x_{0}\right)$
as $h,d\theta\rightarrow0.$ The global scheme defined on $\varOmega$
is consistent if the limit above holds uniformly for all $x\in\varOmega.$
(The domain is assumed to be closed and bounded). \end{defn}
\begin{lem}
\label{lem:consistency1}The consistency holds for (\ref{eq:lambda1dis})
and (\ref{eq:lambda2dis}) and so for (\ref{eq:schemeP2}).\end{lem}
\begin{proof}
Let $x_{0}$ be a reference point with neighbors $x_{1},...,x_{N},$
and direction vectors $v_{i}=x_{i}-x_{0,}$ for $i=1,...,N,$ arranged
symmetrically, if $v_{i}$ is a direction vector, then so is $-v_{i}.$
By Taylor series one has 
\[
\frac{u^{g}\left(x_{0}+v_{i}\right)-2u^{g}\left(x_{0}\right)+u^{g}\left(x_{0}-v_{i}\right)}{\left|v_{i}\right|^{2}}=\frac{d^{2}u^{g}}{dv_{i}^{2}}+O\left(h_{i}^{2}\right).
\]
Let $M$ given symetric $2\times2$ matrix, that we can take it diagonal.
Set $v_{\theta}$ a unit vector. It follows from \cite{oberma08}
(Lemma 3) that 
\[
\min_{\theta\in\left\{ \theta_{1},...,\theta_{N}\right\} }v_{\theta}^{T}Mv_{\theta}=\lambda_{1}+\left(\lambda_{2}-\lambda_{1}\right)O\left(\theta^{2}\right).
\]
 Which implies that
\[
\lambda_{1}\left(\varphi\right)\left(x_{0}\right)-\lambda_{1}^{h,d\theta}\left(\varphi\right)\left(x_{0}\right)=O\left(\bar{h}^{2}+\left(\lambda_{2}-\lambda_{1}\right)d\theta^{2}\right)
\]
and thus consistency holds for (\ref{eq:lambda1dis}). Similar argument
gives consistency for (\ref{eq:lambda2dis}) and so for (\ref{eq:schemeP2}).\end{proof}
\begin{thm}
Suppose that unique viscosity solutions exist for the equation \textcolor{black}{(\ref{eq:P2})}
Then the finite difference scheme given by\textcolor{black}{{} (\ref{eq:schemeP2})}
converges uniformly on compacts subsets of $\varOmega$ to the unique
viscosity solution of the equation.\end{thm}
\begin{proof}
\textcolor{black}{We need to verify consistency and monotonicity.
Consistency follows from Lemma \ref{lem:consistency1} and monotonicity
follows from Lemma \ref{lem:P2 degenerate}.}
\end{proof}
Finally, the scheme yields a fully nonlinear equation defined on grid
functions. We perform the iteration (\ref{eq:FixPoint}) and by Theorem
\ref{thm:fixpoint} will converge to a fixed point which is a solution
of the equation. This approach is used in \cite{oberma08}.

\section{TWO METHODS OF FIXED POINT }

\subsection{The first method (Method B)}

Notice that from Lemma \ref{lem:Lemma equiva} if $u$ is a solution
of (\ref{eq:MAD}) $u\left(x,y\right)=u^{\tilde{g}}\left(x,y\right)$
it follows that $\det\left(D^{2}u\right)=\det\left(D^{2}u^{\tilde{g}}\right),$
where $u^{\tilde{g}}$ is the solution of (\ref{eq:Pgdilde}) for
$\tilde{g}\in L^{2}.$ 

By writing 
\[
\triangle u^{\tilde{g}}=2\sqrt{f}+\tilde{g}=\sqrt{\left(\Delta u^{\tilde{g}}\right)^{2}+2\left(f-\det\left(D^{2}u^{\tilde{g}}\right)\right)}
\]
and expanding $\left(\Delta u^{\tilde{g}}\right)^{2}=\left(u_{xx}^{\tilde{g}}\right)^{2}+\left(u_{yy}^{\tilde{g}}\right)^{2}+2u_{xx}^{\tilde{g}}u_{yy}^{\tilde{g}}$
we have 
\[
\triangle u^{\tilde{g}}=\sqrt{\left(u_{xx}^{\tilde{g}}\right)^{2}+\left(u_{yy}^{\tilde{g}}\right)^{2}+2\left(u_{xy}^{\tilde{g}}\right)^{2}+2f}=2\sqrt{f}+\tilde{g}
\]
Let us define the operator $Q\textrm{ }:\textrm{ }L^{2}\left(\varOmega\right)\rightarrow L^{2}\left(\varOmega\right)$
for $\varOmega\subset\mathbb{R}^{2}$ by 
\[
Q\left(g\right):=\sqrt{\left(u_{xx}^{g}\right)^{2}+\left(u_{yy}^{g}\right)^{2}+2\left(u_{xy}^{g}\right)^{2}+2f}-2\sqrt{f},
\]
with $u^{g}$ solution of $\left(\mathcal{P}^{g}\right).$ So, one
has
\begin{lem}
$\tilde{g}$ is a fixed point of $Q.$ \end{lem}
\begin{proof}
It follows from above expansions. 
\end{proof}

\subsubsection{The scheme}

We consider the following scheme 
\[
g_{n+1}=Q\left(g_{n}\right)=\sqrt{\left(u_{xx}^{g_{n}}\right)^{2}+\left(u_{yy}^{g_{n}}\right)^{2}+2\left(u_{xy}^{g_{n}}\right)^{2}+2f}-2\sqrt{f}.
\]
With initial value $g_{0}>0$ close to zero and $u^{g_{0}}$ is the
solution of $\left(P^{g_{0}}\right).$ 
\begin{rem}
The advantage of this method by comparing it to that in \cite{oberma11}
and \cite{Ben Amou} is that it guarantees, at least, at each iteration
that $tr\left(D^{2}u^{g_{n}}\left(x\right)\right)>0$, which is necessary
to check the convexity.

Although this method turns out to be simple to implement is well suited
in the case where $u^{g}$ is in $H^{2}\left(\varOmega\right).$ If
not, the method may not converge.
\end{rem}

\subsubsection{Algorithm}
\begin{itemize}
\item $g_{0}\geq0$\textcolor{black}{{} (close to 0), solve }$\left(P^{g_{0}}\right),$
($\left(u^{g}\right)^{0}=u^{g_{0}}$\textcolor{black}{being known), }
\item For $n\geq0,$ compute $g^{n+1}$ and $\left(u^{g}\right)^{n+1}$
as follows 
\[
g^{n+1}=Q\left(g^{n}\right),
\]
\[
\left(u^{g}\right)^{n+1}\textrm{ solution of }\left(P^{g^{n+1}}\right).
\]

\end{itemize}
Where, the method involves simply discretising the second derivatives
using standard central differences on a uniform Cartesian grid, as
a result
\begin{eqnarray*}
D_{xx}^{2}u_{ij} & = & \frac{1}{h^{2}}\left(u_{i+1,,j}+u_{i-1,j}-,2u_{i,,j}\right),\\
D_{yy}^{2}u_{ij} & = & \frac{1}{h^{2}}\left(u_{i,,j+1}+u_{i,,j-1}-2u_{i,,j}\right),\\
D_{xx}^{2}u_{ij} & = & \frac{1}{4h^{2}}\left(u_{i+1,,j+1}+u_{i-1,,j-1}-u_{i-1,,j+1}-u_{i+1,,j-1}\right).
\end{eqnarray*}

\subsection{The second method (Method C)}

In the same setting we define the next operator.
\begin{defn}
Let $\varOmega$ a bounded domain in $\mathbb{R}^{2}.$ Define the
operator $F:L^{2}\left(\varOmega\right)\rightarrow L^{2}\left(\varOmega\right),\textrm{ }$
by 
\end{defn}
\begin{equation}
F\left(g\right)=\sqrt{\left|\det\left[D^{2}u^{g}\right]-f\right|}+g,\label{eq:Poinfixe gdilde}
\end{equation}
where $u^{g}$ is a solution of 
\begin{equation}
\left(\mathcal{P}^{g}\right)\left\{ \begin{array}{c}
\Delta u=2\sqrt{f}+g,\\
u_{|\Gamma}=\varphi.
\end{array}\right.
\end{equation}

For $g\in L^{2}\left(\varOmega\right)$, the operator $F$ is well
defined and it is easy to verify that
\begin{lem}
$\widetilde{g}$ is a fixed point of the operator $F.$ \end{lem}
\begin{proof}
Let $u$ a smooth solution of (\ref{eq:MAD}). It follows from Lemma
\ref{lem:Lemma equiva} that $u=u^{\widetilde{g}}.$ Which implies
that $\det\left[D^{2}u^{\widetilde{g}}\right]=\det\left[D^{2}u\right]=f$
and therefore, $F\left(\widetilde{g}\right)=\widetilde{g}.$
\end{proof}

\subsection{The scheme}

We define the following scheme 
\[
g_{n+1}=F\left(g_{n}\right)=\sqrt{\left|\det\left[D^{2}u^{g_{n}}\right]-f\right|}+g_{n}.
\]
With initial value $g_{0}>0$ close to zero and $u^{g_{0}}$ is the
solution of $\left(P^{g_{0}}\right).$ 
\begin{rem}
The method is advantageous, it simply involves evaluating derivatives
and solving the Poisson equation that preserves the convexity constraint. 
\end{rem}

\subsubsection{Algorithm }
\begin{itemize}
\item $g_{0}\geq0$\textcolor{black}{{} (close to 0), solve }$\left(P^{g_{0}}\right),$
($\left(u^{g}\right)^{0}=u^{g_{0}}$\textcolor{black}{being known), }
\item For $n\geq0,$ compute $g^{n+1}$ and $\left(u^{g}\right)^{n+1}$
as follows 
\[
g^{n+1}=\alpha\sqrt{\left|\det\left[D^{2}u^{g^{n}}\right]-f\right|}+g^{n},
\]
with $0<\alpha<1.$
\[
\left(u^{g}\right)^{n+1}\textrm{ solution of }\left(P^{g^{n+1}}\right).
\]

\end{itemize}
As in the above method, second derivatives are descretized using standard
central differences on a uniform Cartesian grid.

\section{Numerical experiments}

The three methods are tested on three different examples (smooth or
singular solutions). The discretization is done in the wide stencil
Finite Difference method with 17- points (see Figure \ref{Stencil 17}).
The number of noeuds meshing is equal to $N*N$ with $N=31,45,63,89,127$,
the step of meshing $h=L/N,$ with  $ L$ is the length of the side
of the rectangular domain $\Omega$. The results obtained are compared
with those in \cite{oberma11}. \\ 

In the \textbf{first example} we study the regular solution given
by : 

\[
u\left(x,y\right)=\exp\left(\frac{\left(x^{2}+y^{2}\right)}{2}\right)\textrm{ with }f\left(x,y\right)=\left(x^{2}+y^{2}+1\right)\exp\left(x^{2}+y^{2}\right).
\]

The Table \ref{table1} summarizes the obtained results for different
meshing. 

\begin{table}
\begin{tabular}{|c|c|c|c|c|}
\hline 
N & Results in \cite{oberma11} & Method A  & Method B & Method C\tabularnewline
\hline 
\hline 
31 & $2.44\times10^{-4}$ & $2.965\times10^{-4}$ & $4.226\times10^{-4}$ & $18\times10^{-4}$\tabularnewline
\hline 
45 & $1.52\times10^{-4}$ & $3.052\times10^{-4}$ & $2.202\times10^{-4}$ & $18\times10^{-4}$\tabularnewline
\hline 
63 & $9.06\times10^{-5}$ & $2.801\times10^{-4}$ & $1.190\times10^{-4}$ & $17\times10^{-4}$\tabularnewline
\hline 
89 & $5.32\times10^{-5}$ & $8.035\times10^{-4}$ & $6.494\times10^{-5}$ & $17\times10^{-4}$\tabularnewline
\hline 
127 & $3.06\times10^{-5}$ & $2.015\times10^{-4}$ & $3.888\times10^{-5}$ & $17\times10^{-4}$\tabularnewline
\hline 
\end{tabular}\caption{Errors $\left\Vert u-u^{N}\right\Vert _{\infty}$ for the exact solution
of the first example on an $N\times N$ grid. We include results from
the wide stencil methods of \cite{oberma11} on seventeen point stencils. }\label{table1}

\end{table}

In Figure \ref{Figure1} we show the surface plot of the solution
and the total CPU time versus $N$ for the methods A, B and C. 
\begin{figure}

\includegraphics[width=6cm,height=5cm]{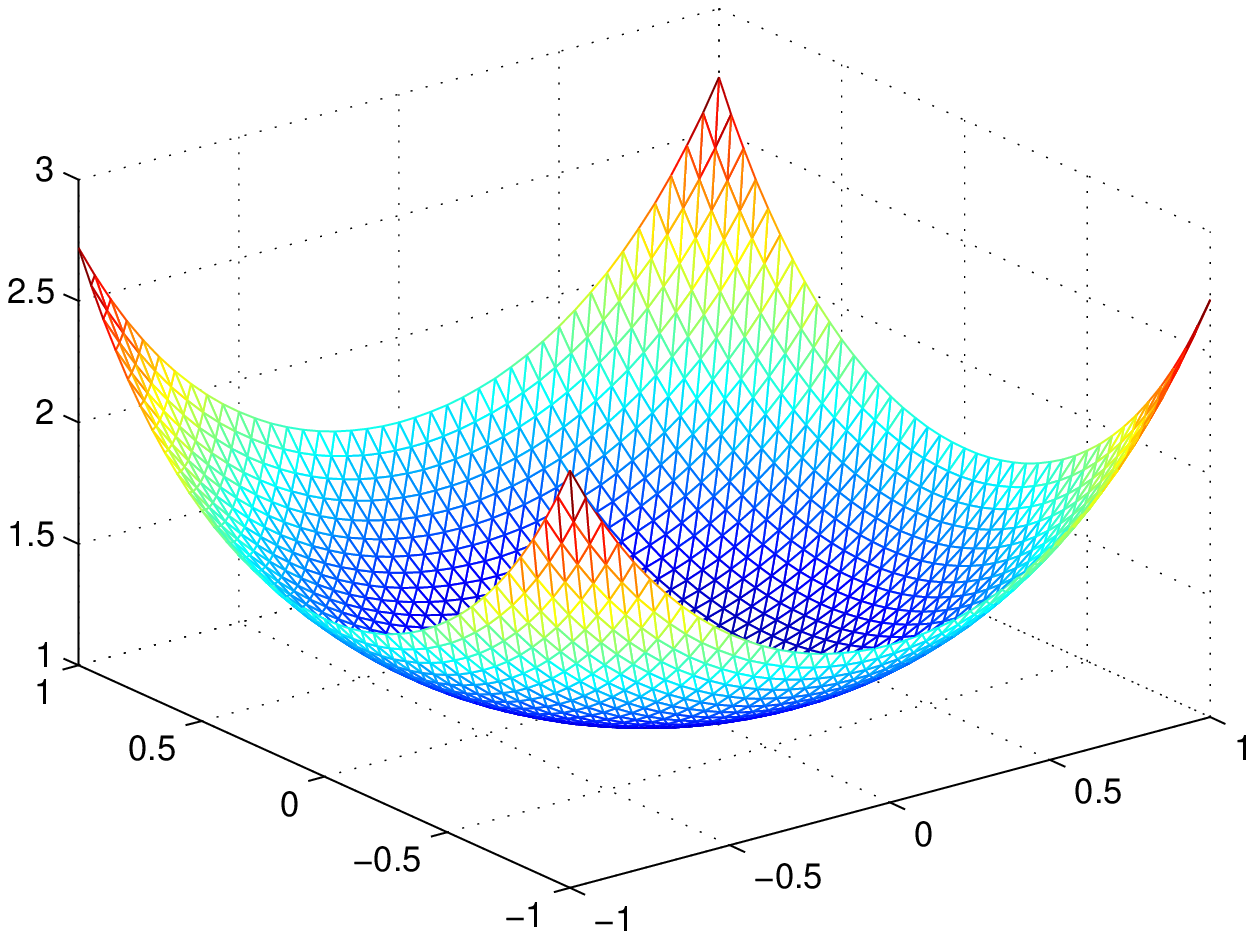}
\label{Figure1}\includegraphics[width=6cm,height=5cm]{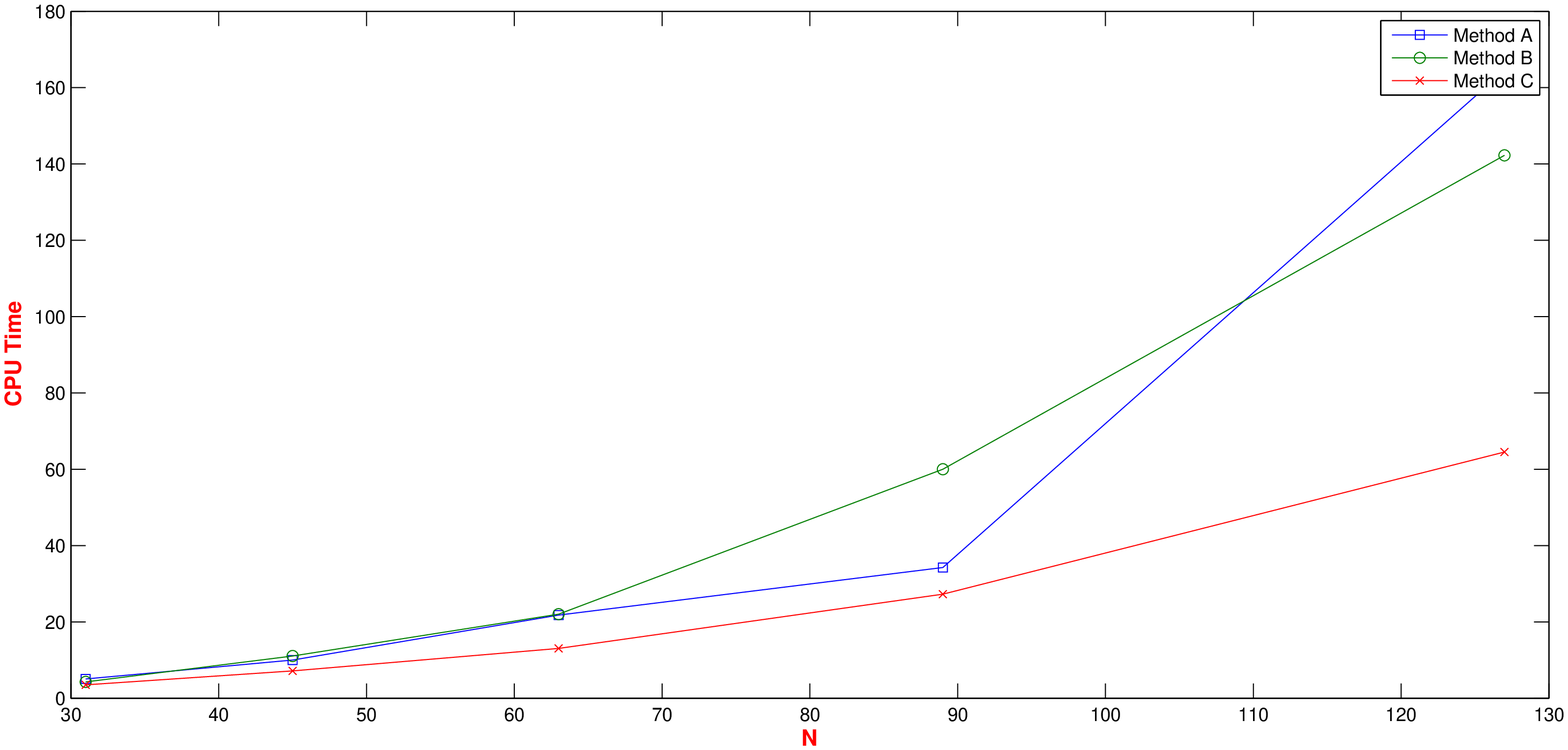}\caption{Results for example 1 on an $N\times N$ grid and total CPU time versus
$N$ for the methods A, B and C.}

\end{figure}

\begin{table}
\begin{tabular}{|c|c|c|c|c|}
\hline 
N & Results in \cite{oberma11} & Method A & Method B & Method C\tabularnewline
\hline 
\hline 
31 & $1.22\times10^{-3}$ & $5.806\times10^{-4}$ & $6.853\times10^{-4}$ & $8.794\times10^{-4}$\tabularnewline
\hline 
45 & $5.9\times10^{-4}$ & $4.92\times10^{-4}$ & $6.719\times10^{-4}$ & $8.727\times10^{-4}$\tabularnewline
\hline 
63 & $4.2\times10^{-4}$ & $4.914\times10^{-4}$ & $2.733\times10^{-4}$ & $8.601\times10^{-4}$\tabularnewline
\hline 
89 & $2.6\times10^{-4}$ & $4.085\times10^{-4}$ & $2.09\times10^{-5}$ & $8.173\times10^{-4}$\tabularnewline
\hline 
127 & $2.0\times10^{-4}$ & $4.056\times10^{-4}$ & $1.08\times10^{-5}$ & $8.164\times10^{-4}$\tabularnewline
\hline 
\end{tabular}\caption{Errors $\left\Vert u-u^{N}\right\Vert _{\infty}$ for the exact solution
of the second example on an $N\times N$ grid. We include results
from the wide stencil methods of \cite{oberma11} on seventeen point
stencils. }\label{table2}
\end{table}

As a \textbf{second example}, which is $C^{1}$ , we take the one
considered in \cite{oberma11} which is given by \[ u(x,y)= \frac{1}{2}((\sqrt{(x-0.5)^2+(y-0.5)^2}-0.2)^{+})^2 \rm \ \ with \rm \ \ f(x,y)=(1-\dfrac{0.2}{ \sqrt{(x-0.5)^{2}+(y-0.5)^{2}}})^{+}. \]

The results are in Table \ref{table2}. 
\begin{figure}

\includegraphics[width=6cm,height=5cm]{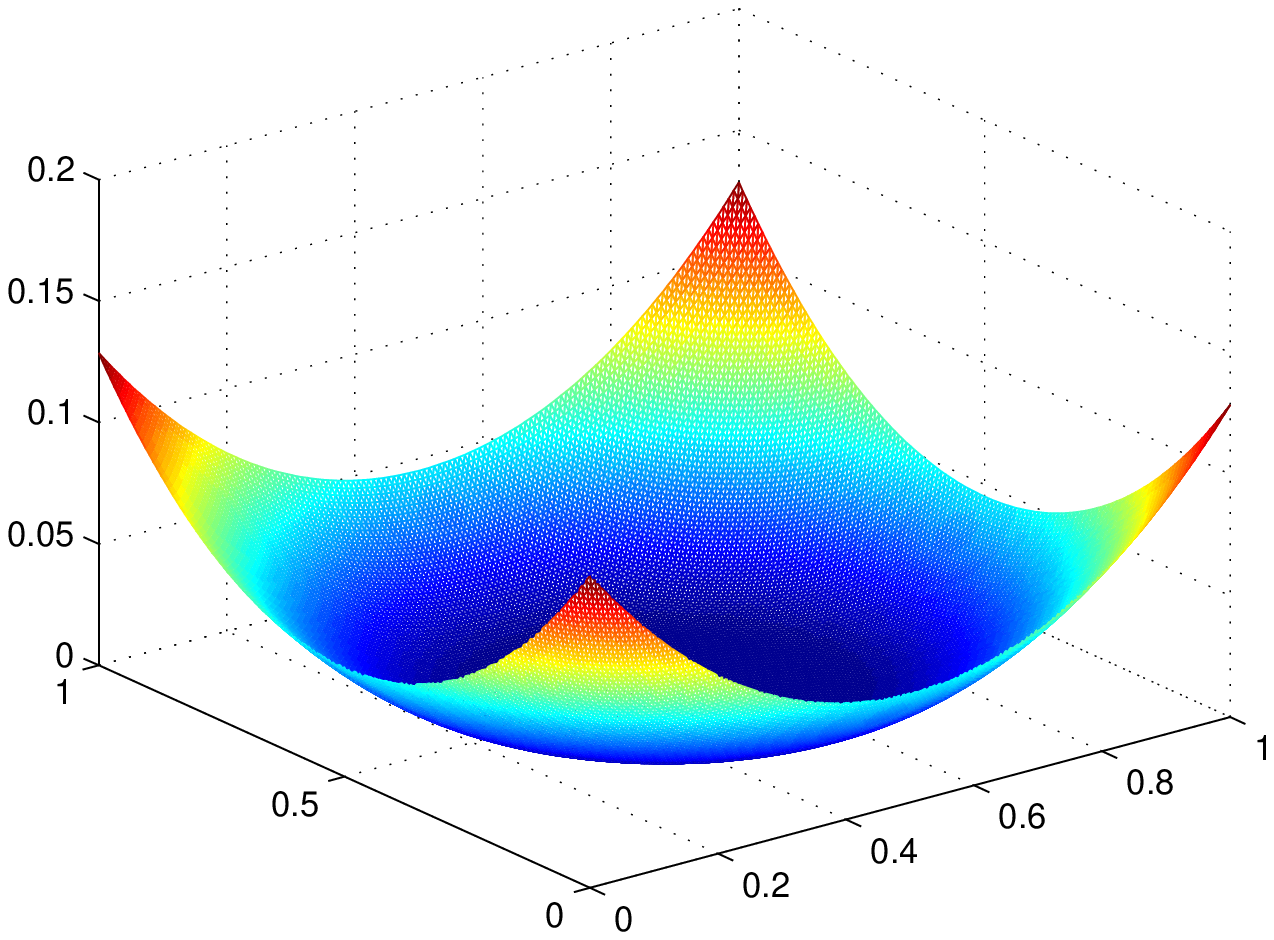}\label{figure2}\includegraphics[width=6cm,height=5cm]{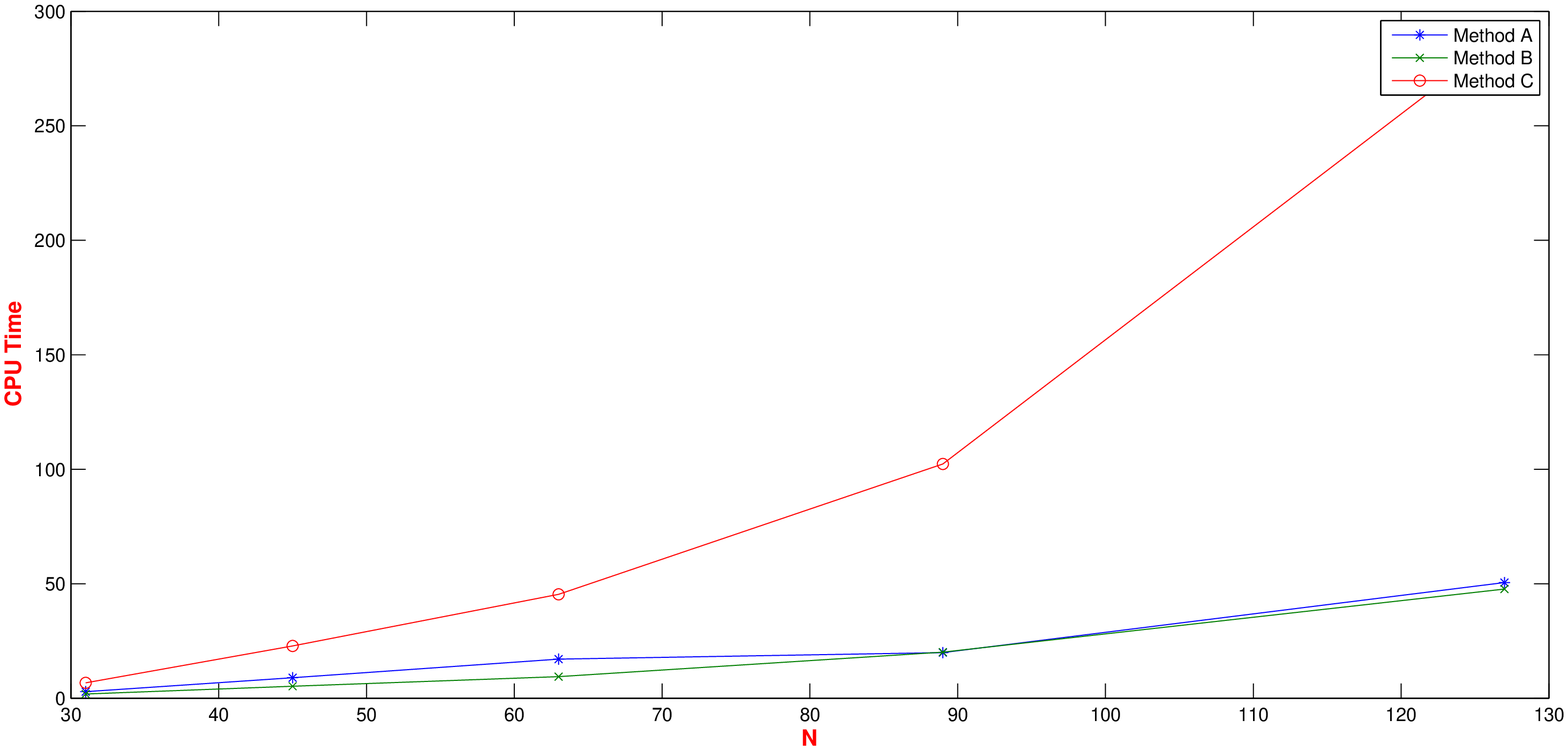}\caption{Results for example 2 on an $N\times N$ grid and total CPU time versus
$N$ for the methods A, B and C.}

\end{figure}
\begin{table}
	\begin{tabular}{|c|c|c|c|c|}
		\hline 
		N & Results in \cite{oberma11} & Method A & Method B & Method C\tabularnewline
		\hline 
		\hline 
		31 & $1.74\times10^{-3}$ & $1.7\times10^{-3}$ & $5.1\times10^{-3}$ & $5.7\times10^{-3}$\tabularnewline
		\hline 
		45 & $9.8\times10^{-4}$ & $1.5\times10^{-3}$ & $4.8\times10^{-3}$ & $5.5\times10^{-3}$\tabularnewline
		\hline 
		63 & $5.9\times10^{-4}$ & $8.9\times10^{-4}$ & $3.9\times10^{-3}$ & $5.5\times10^{-3}$\tabularnewline
		\hline 
		89 & $3.5\times10^{-4}$ & $8.9\times10^{-4}$ & $3.1\times10^{-3}$ & $5.5\times10^{-3}$\tabularnewline
		\hline 
		127 & $2.0\times10^{-4}$ & $8.2\times10^{-4}$ & $2.4\times10^{-3}$ & $5.5\times10^{-3}$\tabularnewline
		\hline 
	\end{tabular}\caption{Errors $\left\Vert u-u^{N}\right\Vert _{\infty}$ for the exact solution
		of the third example on an $N\times N$ grid. We include results from
		the wide stencil methods of \cite{oberma11} on seventeen point stencils. }\label{table3}
\end{table}

Finally, we consider a\textbf{ third example} which is singular at
the bord of the domain $\varOmega=\left[0,1\right]\times\left[0.1\right],$defined
by 

\[ u(x,y)=-\sqrt{(2-x^{2}-y^{2})} \rm \ \ where \ \ f(x,y)=\dfrac{2}{(2-x^{2}-y^{2})^{2}}. \].

The results are illustrated in Table \ref{table3} and

\begin{figure}
\includegraphics[width=6cm,height=5cm]{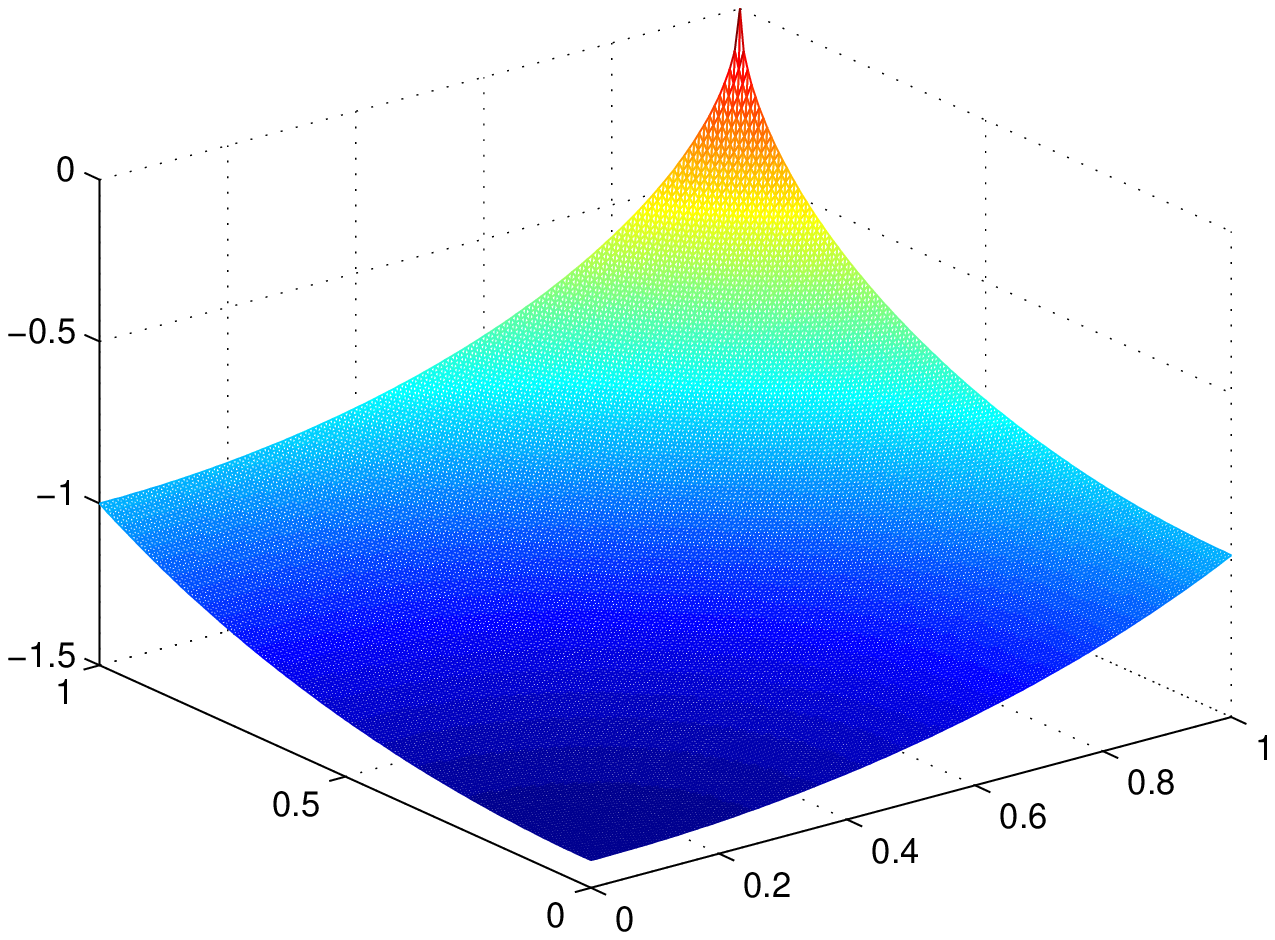}\label{figure3}\includegraphics[width=6cm,height=5cm]{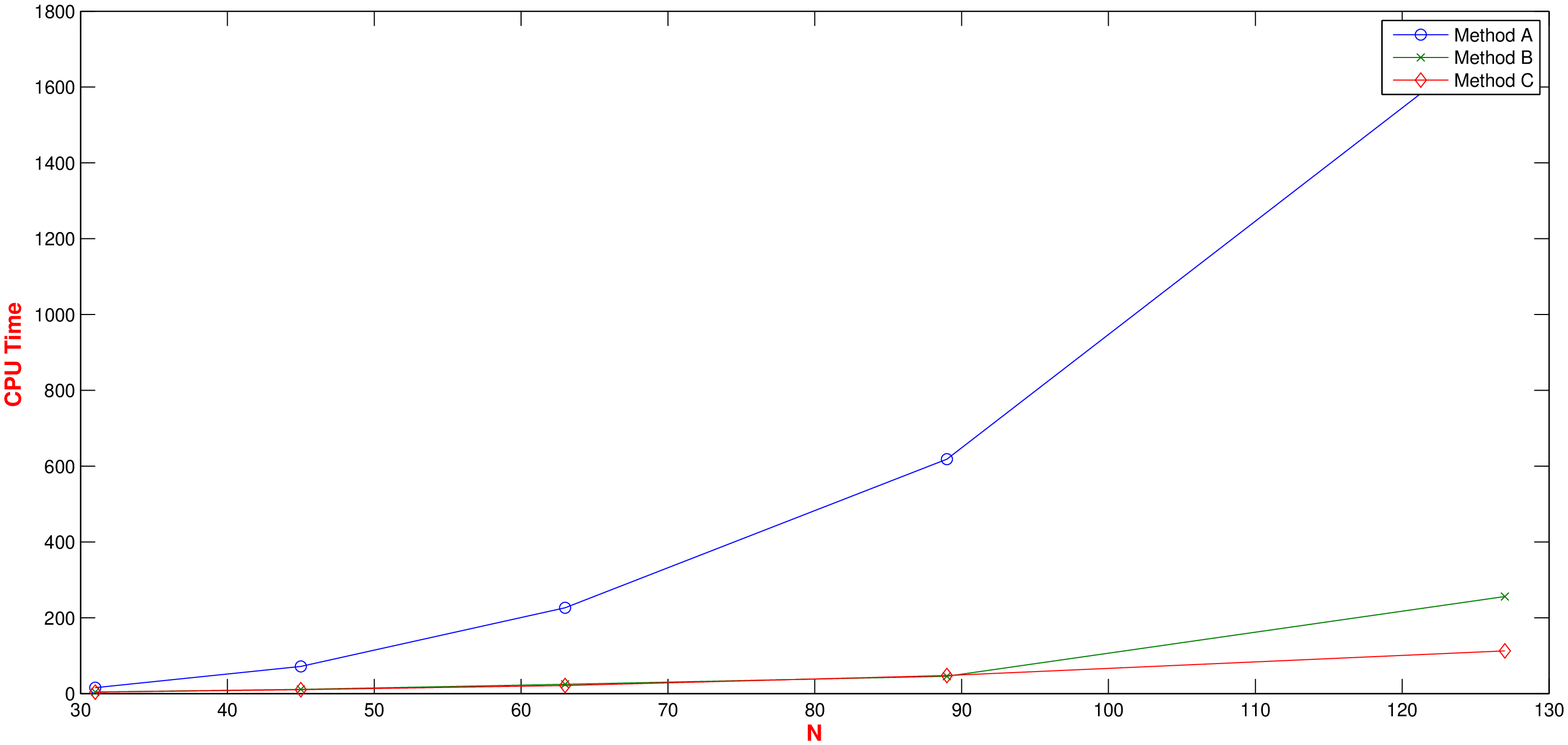}\caption{Results for example 2 on an $N\times N$ grid and total CPU time versus
$N$ for the methods A, B and C.}
\end{figure}

\newpage

\section*{Acknowledgments}

We are indebted to Pr. Pierre-Emmanuelle Jabin for his relevant remarks
and his impressive comments which have greatly improved this work.

.

\begin{thebibliography}{10}
\bibitem{Souganidis}Guy Barles and Panagiotis E. Souganidis. Convergence
of approximation schemes for fully nonlinear second order equations.
Asymptotic Anal., 4(3):271\textendash 283, 1991.

\bibitem{Ben Amou}Jean-David Benamou, Brittany D. Froese, and Adam
M. Oberman. Two numerical methods for the elliptic Monge-Ampère equation.
ESAIM: Math. Model. Numer. Anal., 44(4), 2010.

\bibitem{FBB}Fethi Ben Belgacem, Optimization approach for the Monge-Ampère
equation, Acta Mathematica Scientia, Vol. 38, Issu 4 (2018), 1285-1295.

\bibitem{Budd}C. J. Budd and J. F. Williams. Moving mesh generation
using the parabolic Monge-Ampère equation. SIAM J. Sci. Comput., 31(5):3438\textendash 3465,
2009.

\bibitem{calabi}Eugenio Calabi, Complete affine hyperspheres. I,
Symposia Mathematica, Vol. X (Convegno di Geometria Differenziale,
INDAM, Rome, 1971), Academic Press, London, 1972, pp. 19\textendash 38.
MR0365607 (51 \#1859)

\bibitem{cheng yau}Shiu Yuen Cheng and Shing-Tung Yau, Complete affine
hypersurfaces. I. The com- pleteness of affine metrics, Comm. Pure
Appl. Math. 39 (1986), no. 6, 839\textendash 866, DOI 10.1002/cpa.3160390606.
MR859275 (87k:53127) 

\bibitem{Guti=0000E8rez}Cristian E. Gutiérrez. The Monge-Ampère equation.
Progress in Nonlinear Differential Equations and their Applications,
44. Birkhäuser Boston Inc., Boston, MA, 2001.

\bibitem{Delzano1}\textcolor{black}{G. L. Delzanno, L. Chacón, J.
M. Finn, Y. Chung, and G. Lapenta. An optimal robust equidistribution
method for two-dimensional grid adaptation based on Monge-Kantorovich
optimization. J. Comput. Phys., 227(23):9841\textendash 9864, 2008}

\bibitem{Delzano2}\textcolor{black}{{} J. M. Finn, G. L. Delzanno,
and L. Chacón. Grid generation and adaptation by Monge- Kantorovich
optimization in two and three dimensions. In Proceedings of the 17th
Interna- tional Meshing Roundtable, pages 551\textendash 568, 2008}

\bibitem{Glowinski1}\textcolor{black}{E. J. Dean and R. Glowinski.
An augmented Lagrangian approach to the numerical solution of the
Dirichlet problem for the elliptic Monge-Ampère equation in two dimensions.
Electron. Trans. Numer. Anal., 22:71\textendash 96 (electronic), 2006.}

\bibitem{glowinski2}\textcolor{black}{E. J. Dean and Roland Glowinski.
On the numerical solution of the elliptic Monge- Ampère equation in
dimension two: a least-squares approach. In Partial differential equations,
volume 16 of Comput. Methods Appl. Sci., pages 43\textendash 63. Springer,
Dordrecht, 2008.}

\bibitem{Figali}G. De Philippis and A. Figalli, The Monge-Ampère
equation and its link to optimal trans- portation, Bull. Amer. Math.
Soc. (N.S.) 51 (2014), no. 4, 527\textendash 580. MR3237759

\bibitem{Oliker}T. Glimm and V. Oliker. Optical design of single
reflector systems and the Monge-Kantorovich mass transfer problem.
J. Math. Sci. (N. Y.), 117(3):4096\textendash 4108, 2003. Nonlinear
problems and function theory.

\bibitem{Gr=0000E9goire}Grégoire Loeper and Francesca Rapetti. Numerical
solution of the Monge-Ampére equation by a Newton\textquoteright s
algorithm. C. R. Math. Acad. Sci. Paris, 340(4):319\textendash 324,
2005.

\bibitem{Rehman}T. ur Rehman, E. Haber, G. Pryor, J. Melonakos, and
A. Tannenbaum. 3D nonrigid regis- tration via optimal mass transport
on the GPU. Med Image Anal, 13(6):931\textendash 40, 12 2009.

\bibitem{Haker1}\textcolor{black}{{} Steven Haker, Allen Tannenbaum,
and Ron Kikinis. Mass preserving mappings and image registration.
In MICCAI \textquoteright 01: Proceedings of the 4th International
Conference on Medical Image Computing and Computer-Assisted Intervention,
pages 120\textendash 127, London, UK, 2001. Springer-Verlag}

\bibitem{Haker2}\textcolor{black}{. Steven Haker, Lei Zhu, Allen
Tannenbaum, and Sigurd Angenent. Optimal mass transport for registration
and warping. Int. J. Comput. Vision, 60(3):225\textendash 240, 2004.}

\bibitem{Neilan1}X. Feng and Michael Neilan. Mixed finite element
methods for the fully nonlinear Monge-Ampère equation based on the
vanishing moment method. SIAM J. Numer. Anal., 47(2):1226\textendash 1250,
2009.

\bibitem{oberma11}B. D. Froese and A. M. Oberman, Convergent finite
difference solvers for viscosity solutions of the elliptic Monge-Ampère
equation in dimensions two and higher, SIAM J. Numer. Anal. 49 (2011),
no. 4, 1692\textendash 1714. MR2831067

\bibitem{oberma06}\textcolor{black}{Adam M. Oberman. Convergent difference
schemes for degenerate elliptic and parabolic equations: Hamilton-Jacobi
equations and free boundary problems. SIAM J. Numer. Anal., 44(2):879\textendash 895
(electronic), 2006.}

\bibitem{ober08env}\textcolor{black}{Adam M. Oberman. Computing the
convex envelope using a nonlinear partial differential equation. Math.
Models Methods Appl. Sci., 18(5):759\textendash 780, 2008.}

\bibitem{oberma08}\textcolor{black}{Adam M. Oberman. Wide stencil
finite difference schemes for the elliptic Monge-Ampère equation and
functions of the eigenvalues of the Hessian. Discrete Contin. Dyn.
Syst. Ser. B, 10(1):221\textendash 238, 2008}

\bibitem{oberma10}\textcolor{black}{Adam M. Oberman and Luis Silvestre.
The Dirichlet problem for the convex envelope. Trans. Amer. Math.
Soc. (to appear), 2010 http://arxiv.org/abs/1007.0773}

\bibitem{prusner}V. I. Oliker and L. D. Prussner. On the numerical
solution of the equation (\ensuremath{\partial} 2 z/\ensuremath{\partial}x
2 )(\ensuremath{\partial} 2 z/\ensuremath{\partial}y 2 ) \textminus{}
(\ensuremath{\partial} 2 z/\ensuremath{\partial}x\ensuremath{\partial}y)
2 = f and its discretizations, I. Numer. Math., 54(3):271\textendash{}
293, 1988.

\bibitem{pogorelov}A. V. Pogorelov, On the improper convex affine
hyperspheres, Geometriae Dedicata 1 (1972), no. 1, 33\textendash 46.
MR0319126 (47 \#7672)

\bibitem{Siltakoski}Siltakoski, J. Equivalence of viscosity and weak
solutions for the normalized p(x)-Laplacian. Calc. Var. 57, 95 (2018).
https://doi.org/10.1007/s00526-018-1375-1

\bibitem{Trudingin1}\textcolor{black}{Neil S. Trudinger and Xu-Jia
Wang, The Bernstein problem for affine maximal hypersur- faces, Invent.
Math. 140 (2000), no. 2, 399\textendash 422, DOI 10.1007/s002220000059.
MR1757001 (2001h:53016)}

\bibitem{trudinger2}\textcolor{black}{Neil S. Trudinger and Xu-Jia
Wang, Affine complete locally convex hypersurfaces, Invent. Math.
150 (2002), no. 1, 45\textendash 60, DOI 10.1007/s00222-002-0229-8.
MR1930881 (2003h:53012)}

\bibitem{trudinger3}\textcolor{black}{Neil S. Trudinger and Xu-Jia
Wang, The affine Plateau problem, J. Amer. Math. Soc. 18 (2005), no.
2, 253\textendash 289, DOI 10.1090/S0894-0347-05-00475-3. MR2137978
(2006e:53071)}

\bibitem{zheligovsky}V. Zheligovsky, O. Podvigina, and U. Frisch.
The Monge-Ampère equation: Various forms and numerical solution. J.
Comput. Phys., 229(13):5043\textendash 5061, 2010.

\end{thebibliography}
\end{document}